\documentclass[a4paper,11pt,oneside]{article}
\usepackage{blindtext}
\usepackage{xcolor}
\usepackage[latin1]{inputenc}
\usepackage[russian,english]{babel}
\usepackage{amsmath,amssymb}
\usepackage{textcomp}
\usepackage{bbm}
\usepackage[T2A,T1]{fontenc}
\usepackage{amsthm}
\usepackage[all,cmtip]{xy}
\usepackage{mathtools}
\usepackage{tikz}
\usepackage{tikz-cd}
\usetikzlibrary{matrix}
\usepackage{tocbibind}
\usepackage{graphicx}
\usepackage{wrapfig}
\usepackage{mathrsfs}
\usepackage{fancyhdr}
\usepackage[colorinlistoftodos]{todonotes}
\usetikzlibrary{arrows,decorations.pathmorphing}
\title{}
\author{}
\date{}
\usetikzlibrary{calc,trees,positioning,arrows,chains,shapes.geometric,%
	decorations.pathreplacing,decorations.pathmorphing,shapes,%
	matrix,shapes.symbols}
\usepackage{titlesec}
\usepackage{graphicx}

\titleformat{\chapter}[display]
{\bfseries\huge}
{\filcenter\MakeUppercase{\chaptertitlename} \Huge\thechapter}
{1ex}
{\titlerule\vspace{1ex}\filcenter}
[\vspace{1ex}\titlerule]
\usepackage{scrextend}

\newtheorem{thm}{Theorem}

\newtheorem{prop}{Proposition}

\newtheorem{coroll}{Corollary}

\newtheorem{lemma}{Lemma}

\newcommand{\Z}{\mathbb{Z}}
\newcommand{\Co}{\mathbb{C}}
\newcommand{\R}{\mathbb{R}}
\newcommand{\N}{\mathbb{N}}
\newcommand{\Q}{\mathbb{Q}}

\newcommand{\Qal}{\overline{\Q}}

\DeclareMathOperator{\Pm}{P}

\theoremstyle{plain} 
\newcommand{\thistheoremname}{}
\newtheorem{genericthm}[thm]{\thistheoremname}

\newtheorem*{genericthm*}{\thistheoremname}
\newenvironment{namedthm*}[1]
{\renewcommand{\thistheoremname}{#1}%
	\begin{genericthm*}}
	{\end{genericthm*}}

\DeclareMathOperator{\ord}{ord}

\DeclareMathOperator{\ind}{ind}

\DeclareMathOperator{\paruno}{(}
\DeclareMathOperator{\pardue}{)}

\DeclareMathOperator{\modulus}{mod}

\DeclareMathOperator{\h}{ht}

\makeatletter
\newcommand*{\math@version@bold}{bold}
\DeclareMathOperator\Sha{
	\textrm{%
		\usefont{T2A}{cmr}{\ifx\math@version\math@version@bold bx\else m\fi}{n}%
		\CYRSH
	}%
}
\makeatother

\newcommand\blfootnote[1]{%
	\begingroup
	\renewcommand\thefootnote{}\footnote{#1}%
	\addtocounter{footnote}{-1}%
	\endgroup
}

\begin{document}
	\begin{center}
		\textbf{TOWARDS A GENERALIZATION OF THE VAN DER WAERDEN'S CONJECTURE FOR $S_n$-POLYNOMIALS WITH INTEGRAL COEFFICIENTS OVER A FIXED NUMBER FIELD EXTENSION}
	\end{center}
\begin{center}
	Ilaria Viglino
\end{center}
\begin{center}
	\textbf{Abstract}
	
\end{center}
\begin{addmargin}[2em]{2em}

\fontsize{10pt}{12pt}\selectfont
The van der Waerden's Conjecture states that the set $\mathscr{P}_{n,N}^0(\Q)$ of monic integer polynomials $f(X)$ of degree $n$, with height $\le N$ such that the Galois group $G_{K_f/\Q}$ of the splitting field $K_f/\Q$ is the full symmetric group, has order $|\mathscr{P}_{n,N}^0(\Q)|=(2N)^n+O_n(N^{n-1})$ as $N\rightarrow+\infty$. The conjecture has been shown previously for cubic and quartics polynomials by van der Waerden, Chow and Dietmann. Subsequently, Bhargava proved it for $n\ge6$. In this paper, we generalize the result for polynomials with coefficients in the ring of algebraic integers $\mathcal{O}_K$ of a fixed finite extension $K/\Q$ of degree $d$, for some values of $n$ and $d$.
\end{addmargin}
\normalsize
\section{Introduction}
Fix a field extension $K/\Q$ of degree $d$. Let $n\ge2$ and let $N$ be positive integers. We consider monic polynomials with coefficients in $\mathcal{O}_K$:$$f(X)=X^n+\alpha_{n-1}X^{n-1}+\dots+\alpha_0.$$Choose an ordered integral basis $(\omega_1,\dots,\omega_d)$ of $\mathcal{O}_K$ over $\Z$.\blfootnote{The main results of this paper are contained in my Ph.D. thesis. One can use them to count the number of $S_n$-polynomials congruent to a given polynomial modulo a prime ideal.} We have, for all $k=0,\dots,n-1$,$$\alpha_k=\sum_{i=1}^{d}a_{i}^{(k)}\omega_i$$for unique $a_{i}^{(k)}\in\Z$. We view the coefficients $\alpha_0,\dots,\alpha_{n-1}$ as independent, identically distribuited random variables taking values uniformely in $[-N,N]^{d}$. Define the \textit{height} of $\alpha_k$ as $\mbox{ht}(\alpha_k)=\max_{i}|a_{i}^{(k)}|$ and the \textbf{height} of the polynomial $f$ to be$$\mbox{ht}(f)=\max_{k}\mbox{ht}(\alpha_{k}).$$For $n\ge2$, $N>0$ define$$\mathscr{P}_{n,N}^{0}(K)=\{f\in\mathcal{O}_K[X]:\mbox{ht}(f)\le N,\ G_{K_f/K}\cong S_n\},$$where $K_f$ is the splitting field of $f$ over $K$ inside a fixed algebraic closure $ \Qal $ of $ \Q $. We call these polynomials \textbf{$S_n$-polynomials over $K$}, or simply $S_n$-polynomials when there is no need to specify the base field.\\

It has been proven that almost all polynomials are $S_n$-polynomials in the following sense:$$\frac{|\mathscr{P}_{n,N}^{0}(K)|}{|\mathscr{P}_{n,N}(K)|}\underset{N\rightarrow+\infty}{\longrightarrow}1.$$For instance, in the case $K=\Q$, Van der Waerden gave in \cite{Wa} an explicit error term $O(N^{-1/6})$. It has improved in \cite{Gal} using large sieve to $ O(N^{-1/2}\log N) $, and more recently by Dietmann \cite{Di} using resolvent polynomials to $O(N^{-2+\sqrt{2}+\varepsilon})$ for every $\varepsilon>0$. The best estimate can be found in \cite{Bh1}, who proved the following result, conjectured by van der Waerden.
\begin{namedthm*}{Theorem}[Bhargava]
	If either $n=3,4$ or $n\ge6$, one has, $$|\mathscr{P}_{n,N}^{0}(\Q)|=(2N)^n+O(N^{n-1}),$$as $ N\rightarrow\infty $.
\end{namedthm*}

\subsection{Main results}
In his work, Bhargava used a combination of algebraic techniques and Fourier analysis over finite fields. In the below theorem, we generalize this result for polynomials in $\mathscr{P}_{n,N}^{0}(K)$, for certain values of $n$ and $d$ (Section 2.3). We also use large sieve over number fields to prove the upper bound for all $n\ge3$ and $d\ge1$ (Section 2.2).
\begin{thm}
	Let $d\ge1$ and $n\ge2$. There exist positive constants $\theta$ and $\theta_n$ such that the number of non $S_n$-polynomials is$$|\mathscr{P}_{n,N}(K)\setminus\mathscr{P}_{n,N}^{0}(K)|\ll_{n,K}N^{d(n-\theta)}(\log N)^{\theta_n},$$as $N\rightarrow+\infty$. In particular,\begin{enumerate}
		\item[$\paruno 1\pardue$] if $n=2$, we can choose $\theta=1$, $\theta_2=1$;
		\item[$\paruno 2\pardue$] for all $d\ge1$ and $n\ge3$ the above estimate holds with $\theta=1/2$ and $\theta_n=1-\gamma_n$, where $\gamma_n\sim(2\pi n)^{-1/2}$;
		\item[$\paruno 3\pardue$] if moreover $n\ge[3d+\sqrt{9d^2-4d}]+1$, we can take $\theta=1$ and $\theta_n=0$;
		\item[$\paruno 4\pardue$] if $d=1$ the result with $\theta=1$ and $\theta_n=0$ also holds for $n=3,4$.
	\end{enumerate}
\end{thm}
Let $G$ be a subgroup of $S_n$; define$$\mathscr{N}_n(N,G;K)=\mathscr{N}_n(N,G)=\{f\in\mathscr{P}_{n,N}(K):G_f\cong G\},$$and $N_n(N,G;K)=N_n(N,G)=|\mathscr{N}_n(G,N)|$. Theorem 1 states that\begin{equation}
	N_n(N,G)\ll_{n,K} N^{d(n-\theta)}(\log N)^{\theta_n}
\end{equation}as $N\rightarrow+\infty$ for all $G\subset S_n$.

Part (4) of Theorem 1 is the van der Waerden's conjecture for cubic and quartic fields, which has already been proven by Chow and Dietmann in \cite{CD}. 
Recently Bhargava \cite{Bh1} proved the conjecture for all $n\ge 6$ and $K=\Q$. Part (3) of Theorem 1 is a generalization of this result for polynomials with integral coefficients in a number field $K$, for some values of $d$ and $n$. Finally, for part (2) we apply large sieve to the set $\mathscr{P}_{n,N}$ (see \cite{Gal} for the analogous result for $d=1$).\\

In Chapter 3, we are going to count the number of $f\in\mathscr{P}_{n,N}$, whose Galois group is a fixed transitive subgroup of $S_n$. In order to prove the following theorem, we use resolvents of polynomials, the discriminant variety and generalize some results of \cite{Di}, \cite{Di2}, \cite{HB} and \cite{BHB}.

\begin{thm}
	For every $\varepsilon>0$ and positive integer $n$,\begin{multline*}
		|\{(\alpha_0,\dots,\alpha_{n-1})\in\mathcal{O}_K^n:\h(\alpha_j)\le N\ \forall j,\\
		f(X)=X^n+\alpha_{n-1}X^{n-1}+\dots+\alpha_0\mbox{ has }G_{K_f/K}=G\}|\\ \ll_{n,d,\varepsilon}N^{d(n-1+1/[S_n:G])+\varepsilon},
	\end{multline*}where $[S_n:G]$ is the index of $G$ in $S_n$.
\end{thm}

\section{Proof of Theorem 1}
\subsection{Counting reducible polynomials over $K$}
Firstly, we prove (1) in the case $G$ intransitive subgroup of $S_n$. The polynomials having such $G$ as Galois group are exactly those that factor over $K$.\\
Let $1\le k\le n/2$ and let$$\rho_k(n,N;K)=\rho_k(n,N)=\{f\in\mathscr{P}_{n,N}(K):f \mbox{ has a factor of degree }k \mbox{ over }K\}$$and$$\rho(n,N;K)=\rho(n,N)=\{f\in\mathscr{P}_{n,N}(K):f \mbox{ reducible } \mbox{over }K\}.$$
\begin{prop}
	One has$$\rho_k(n,N)\ll_{n,K}\begin{cases}
		N^{d(n-k)}&\mbox{if }k<n/2\\
		N^{d(n-k)}\log N&\mbox{if }k=n/2.
	\end{cases}$$In particular, if $n\ge3$,$$\rho(n,N)\ll_{n,K}N^{d(n-1)},$$as $N\rightarrow+\infty$.
\end{prop}
\noindent
Note that $\rho(n,N)\ll_{n,K}N^{d(n-1)}\log N$ if $n=2$, which proves Theorem 1, part (1).
\begin{lemma}
	Let $\beta\in K$ be a root of $f(X)=X^n+\alpha_{n-1}X^{n-1}+\dots+\alpha_0\in\mathcal{O}_K[X]$ of height $N$. Then$$|N_{K/\Q}(\beta)|\ll_{n,K}N^d.$$
\end{lemma}
\begin{proof}
	This follows from the analogous results for polynomials with coefficients over $\Co$ (see \cite{Di2}, Lemma 1). For all $i=1,\dots,d$, $\sigma_i(\beta)$ is a complex root of $\sigma_i\circ f\in\Co[X]$, hence$$|\sigma_i(\beta)|\le\frac{1}{\sqrt[n]{2}-1}\max_{1\le k\le n} \left| \frac{\sigma_i(\alpha_{n-k})}{\binom{n}{k}}\right| ^{1/k}.$$Then$$|N_{K/\Q}(\beta)|\le\left( \frac{1}{\sqrt[n]{2}-1}\right) ^d\prod_{i=1}^{d}\max_{1\le k\le n}|\sigma_i(\alpha_{n-k})|^{1/k}\ll_{n,d}N^d.$$
\end{proof}
By Proposition 1, it follows that$$\underset{\tiny\mbox{intransitive}}{\sum_{G\subset S_n}}N_n(N,G)=\rho(n,N)\ll_{n,K}N^{d(n-1)}$$for all $n\ge3$, as $N\rightarrow+\infty$.
\begin{proof}(Proposition 1) Assume that $f(X)=X^n+\alpha_{n-1}X^{n-1}+\dots+\alpha_0$ of height $\le N$ factors over $K$ as $f(X)=g(X)h(X)$, where\begin{align*}
		g(X)&=X^q+a_{q-1}X^{q-1}+\dots+a_0;\\
		h(X)&=X^r+b_{r-1}X^{r-1}+\dots+b_0,
	\end{align*}where $n=q+r$. We call this set of $f$'s $\rho_{q,r}(n,N)$. We therefore have to find an upper bound for the number of coefficients of $g$ and $h$ so that $f=gh$ and $\h(\alpha_i)\le N$ for all $i=0,\dots,n-1$.\\
	By Knonecker's theorem, every product $\zeta=a_ib_j$ is a root of an equation of the form$$\zeta^m+d_1\zeta^{m-1}+\dots+d_m=0,$$where $d_j=d_j(\alpha_0,\dots,\alpha_{n-1})$ is homogeneous of degree $j$ in the coefficients of $f$.\\
	Let $\sigma_i:K\hookrightarrow\Co$, $i=1,\dots,d$ be the $\Q$-embeddings of $K$ into $\Co$. In particular, if $\alpha\in\mathcal{O}_K$, $\alpha=\sum_{k=1}^{d}a_k\omega_k$ has height $\le N$, then for all $k=0,\dots,n-1$,$$|\sigma_i(\alpha)|\le C_{K,i} N$$for all $i=1,\dots,d$, where $C_{K,i}=\sum_{j=1}^{d}|\sigma_i(\omega_j)|$. Hence\begin{align*}
		|N_{K/\Q}(\alpha)|&=\Big|\prod_{i=1}^{d}\sigma_i(\alpha)\Big|\\
		&=\prod_{i=1}^{d}\Big|\sum_{i=1}^{d}\sigma_i(\omega_j)a_j\Big|\\
		&\le C_K N^d,
	\end{align*}where $C_K=\sum_{i=1}^{d}C_{K,i}$. It follows that since $\h(d_j)\ll_{n,K} N^j$,$$N_{K/\Q}(d_j)\ll_{n,K}N^{dj}$$for all $j$. Now,$$\Big(\frac{\zeta}{N}\Big)^m+\frac{d_1}{N}\Big(\frac{\zeta}{N}\Big)^{m-1}+\dots+\frac{d_m}{N^m}=0,$$hence $\frac{\zeta}{N}$ is a root of an equation with coefficients of norm$$N_{K/\Q}\Big(\frac{d_j}{N^j}\Big)\ll_{n,K}1.$$As in Lemma 1, one has $N_{K/\Q}\Big(\frac{\zeta}{N}\Big)\ll_{n,K,q,r}1,$ so$$N_{K/\Q}(\zeta)=N_{K/\Q}(a_ib_j)\ll_{n,K,q,r}N^d$$ for all $i,j$. Let\begin{align*}
		A&=\max_{i}|N_{K/\Q}(a_i)|;\\
		B&=\max_{j}|N_{K/\Q}(b_j)|.
	\end{align*}By the above$$AB\ll_{n,K,q,r}N^d.$$According to the Wiener-Ikehara Tauberian theorem, the number of principal ideals of norm $\le x$ is $\ll x$. Given $A,B$ sufficiently large, there are at most $\ll_{n,K}AqA^{q-2}=qA^{q-1}$ polynomials $g$ and $\ll_{n,K}rB^{r-1}$ polynomials $h$, since at least one coefficient of $g$ has norm $A$ ($q$-possibilities), the remaining $q-1$ have norm $\le A$, and the same for $h$. It total, for $A,B$ large enough the number of products $gh$ is at most$$\ll_{n,K}qrA^{q-1}B^{r-1}.$$It turns out that\begin{align*}
		\rho_{q,r}(n,N)&\ll_{n,K,q,r}qr\sum_{AB\ll N^d}A^{q-1}B^{r-1}\\
		&\ll_{n,K,q,r}\sum_{A\ll N^d}A^{q-1}\sum_{B\ll N^d/A}B^{r-1}\\
		&\ll\sum_{A\ll N^d}A^{q-1}\Big(\frac{N^d}{A}\Big)^r\\
		&=N^{dr}\sum_{A\ll N^d}A^{q-r-1}.
	\end{align*}We can assume $q\le r$.
	\begin{enumerate}
		\item[$\bullet$] If $q<r$, the last sum is convergent, so$$\rho_{q,r}(n,N)\ll_{n,K,q,r}N^{dr}$$as $N\rightarrow+\infty$.
		\item[$\bullet$] If $q=r$,\begin{align*}
			\rho_{q,r}(n,N)&\ll_{n,K,q,r}N^{dr}\sum_{A\ll N^d}\frac{1}{A}\\
			&\ll_{n,K,q,r}N^{dr}\log N
		\end{align*}as $N\rightarrow+\infty$.
	\end{enumerate}
\end{proof}

In fact, we go further by extending a result of Chela \cite{Ch} and proving an asymptotic for $\rho(n,N;K)$.
\begin{thm}
	Let $n\ge3$. Then$$\lim_{N\rightarrow+\infty}\frac{\rho(n,N;K)}{N^{d(n-1)}}=2^{d(n-1)}\left( D_{n,K}\cdot\left( \frac{C_KC_K'}{h_K}\right)^{n-1} +1+\frac{A_{n,K}k_{n,d}}{2^{d(n-1)}}\right), $$where $A_{n,K}$ is an explicit constant, $C_K$ is the residue at 1 of $\zeta_K$, $h_K$ is the class number of $K$,\begin{align*}
		D_{n,K}&=\underset{1<|N_{K/\Q\nu}|<C_K'N^d}{\sum_{\nu\in\mathcal{O}_K}}\frac{1}{|N_{K/\Q}\nu)|^{n-1}},\\
		C_K'&=\prod_{j=1}^{d}\Big|\sum_{k=1}^{d}\sigma_j(\omega_k)\Big|,\\
		k_{n,d}&=\mbox{vol}(R)=\int_{R}\dots\int dy_1^{(0)}\dots dy_1^{(n-2)}\dots dy_d^{(0)}\dots dy_d^{(n-2)},
	\end{align*}where $R$ is the region of the $d(n-1)$-dimensional Euclidean space defined by\begin{align*}
		|y_k^{(j)}|\le1\ \ \ \forall j,k,&&\Big|\sum_{j=0}^{n-2}y_k^{(j)}\Big|\le1\ \ \ \forall j,k.
	\end{align*}
\end{thm}We assume from now on that $n\ge3$. By Proposition 1 and by definition of $\rho,\rho_k$ if follows that $$\underset{N}{\lim}\frac{\rho(n,N)}{N^{d(n-1)}}=\underset{N}{\lim}\frac{\rho_1(n,N)}{N^{d(n-1)}}.$$So we reduce to prove the asymptotic for$$\frac{\rho_1(n,N)}{N^{d(n-1)}}$$as $N\rightarrow+\infty$.

Let $\nu\in\mathcal{O}_K$ and let$$T_{n,N}(\nu;K)=T_{n,N}(\nu):=\{f\in\in\mathscr{P}_{n,N}(K):f\mbox{ has a linear factor }X+\nu\}.$$
\begin{lemma}
	One has $$\rho_1(n,N)-\underset{|N_{K/\Q}\nu|\ll_{n,K} N^d}{\sum_{\nu\in\mathcal{O}_K}}T_{n,N}(\nu)=o(N^{d(n-1)})$$as $N\rightarrow+\infty$.
\end{lemma}

\begin{proof}
	Note that $\sum_\nu T_{n,N}(\nu)\ge\rho_1(n,N)$, since in the first sum a polynomial may be counted repeatedly. Let $R_i$ be the number of $f\in\mathscr{P}_{n,N}(K)$ with exactly $i$ distinct linear factors, and let $\rho_1'(n,N)$ be the number of $f\in\mathscr{P}_{n,N}(K)$ with two linear factors (not necessarily distinct).\\
	Each of the $R_i$ is counted in $\sum_\nu T_{n,N}(\nu)$ exactly $i$ times. Moreover for $i>1$,$$R_i\le\rho_1'(n,N)<\rho_2(n,N).$$By Proposition 1, $\rho_2(n,N)=o(N^{d(n-1)})$, therefore $\rho_1(n,N)$ and $\sum_\nu T_{n,N}(\nu)$ differ in a $o(N^{d(n-1)})$ term.
\end{proof}

\begin{lemma}
	One has$$\lim_{N\rightarrow+\infty}\underset{1<|N_{K/\Q}\nu|\le C_K' N^d}{\sum_{\nu\in\mathcal{O}_K}}\frac{T_{n,N}(\nu)}{N^{d(n-1)}}=2^{d(n-1)}D_{n,K}\cdot\left( \frac{C_KC_K'}{h_K}\right)^{n-1},$$where $D_{n,K}\le\zeta_K(n-1)$.
\end{lemma}
\begin{proof}
	Since $T_{n,N}(\nu)=T_{n,N}(\nu')$ if $N_{K/\Q}\nu=N_{K/\Q}\nu'$, we can assume that $2\le N_{K/\Q}\nu\le C_K'N^d$. A polynomial $f\in\mathscr{P}_{n,N}$ with a linear factor $X+\nu$ is of the form\begin{equation}
		f(X)=(X+\nu)(X^{n-1}+\beta_{n-2}X^{n-2}+\dots+\beta_0)
	\end{equation}for some $\beta_j\in\mathcal{O}_K$ for all $j$. Thus $T_{n, N}(\nu)$ is equal to the number of $(n-1)$-tuples $(\beta_{n-2},\dots,\beta_0)\in\mathcal{O}_K^{n-1}$ satisfying (2) for $f$ of height $\le N$. We get\begin{equation}
		\begin{cases}
			\beta_0=\frac{\alpha_0}{\nu}&\\
			\beta_i=\frac{\alpha_i-\beta_{i-1}}{\nu}&i=1,\dots,n-2\\
			\alpha_{n-1}=\beta_{n-2}+\nu.&
		\end{cases}
	\end{equation}Write $\alpha_i=\sum_{k=1}^{d}a_k^{(i)}\omega_k$ and $\beta_i=\sum_{k=1}^{d}b_k^{(i)}\omega_k$ for all $i$, where $a_k^{(i)},b_k^{(i)}\in\Z$ for all $i,k$. Now fix $\beta_{i-1}$ and let $\alpha_i$ varies with $\h(\alpha_i)\le N$. One gets from (3),\begin{align*}
		N_{K/\Q}\beta_i&=\prod_{j=1}^{d}\sigma_j(\beta_i)\\
		&=\prod_{j=1}^{d}\sum_{k=1}^{d}(a_k^{(i)}-b_k^{(i-1)})\sigma_j(\omega_k)\cdot\frac{1}{N_{K/\Q}\nu}.
	\end{align*}Once fixed $\beta_{i-1}$ (i.e. $b_k^{i-1}$ for all $k$), the norm of $\beta_i$ lies in an interval of amplitude$$C_K'\frac{(2N)^d}{N_{K/\Q}\nu},$$where $C_K'=\prod_{j=1}^{d}\Big|\sum_{k=1}^{d}\sigma_j(\omega_k)\Big|.$ By definition of the ideal class group of $K$, the set of principal ideals of $\mathcal{O}_K$ is the identity element.
	Let $L$ denote the average over $m$ of the number of principal ideals of norm $m$. The uniform distribuition of the ideals among the $h_K$ ideal classes of $\mathcal{O}_K$ and the Wiener-Ikehara Tauberian theorem imply that$$\frac{1}{h_K}\sum_{m\le x}|\{I\subseteq\mathcal{O}_K:N(I)=m\}|\sim\frac{1}{h_K}C_Kx\sim Lx;$$hence$$L=\frac{C_K}{h_K},$$where $C_K$ is the residue at 1 of $\zeta_K$. Therefore there are$$\left[ \frac{C_KC_K'}{h_K}\frac{(2N)^d}{N_{K/\Q}\nu}\right]\mbox{ or } \left[ \frac{C_KC_K'}{h_K}\frac{(2N)^d}{N_{K/\Q}\nu}\right]+1$$integral elements $\beta_i$. The total number of solutions of the second equation of (3) is of the form$$\prod_{i=1}^{n}\left(\frac{C_KC_K'}{h_K}\frac{(2N)^d}{N_{K/\Q}\nu}+r_{\nu,i} \right), $$where $r_{\nu,i}=0$ or $1$. By induction$$\beta_{n-2}=\frac{\alpha_{n-2}}{\nu}-\frac{\alpha_{n-1}}{\nu^2}+\dots+(-1)^{n-2}\frac{\alpha_0}{\nu^{n-1}},$$from which$$|N_{K/\Q}\beta_{n-2}|\le C_K'N^d\left(\frac{1}{N_{K/\Q}\nu}+\frac{1}{(N_{K/\Q}\nu)^2}+\dots+\frac{1}{(N_{K/\Q}\nu)^{n-1}} \right). $$So for $\nu\in\mathcal{O}_K$ with $2\le N_{K/\Q}\nu<C_K'N^d$, the values of $\beta_{n-2}$ also satisfy the third equation in (3) provided $N$ is large enough. We have therefore\begin{multline*}
		\underset{1<|N_{K/\Q}\nu|\le C_K' N^d}{\sum_{\nu\in\mathcal{O}_K}}T_{n,N}(\nu)=\underset{2\le N_{K/\Q}\nu< C_K' N^d}{\sum_{\nu\in\mathcal{O}_K}}H_K(\nu)\cdot T_{n,N}(\nu)\\+\underset{N_{K/\Q}\nu= C_K' N^d}{\sum_{\nu\in\mathcal{O}_K}}H_K(\nu)\cdot T_{n,N}(\nu)\\
		=\underset{2\le N_{K/\Q}\nu< C_K' N^d}{\sum_{\nu\in\mathcal{O}_K}}H_K(\nu)\cdot\prod_{i=1}^{n-1}\left(\frac{C_KC_K'}{h_K}\frac{(2N)^d}{N_{K/\Q}\nu}+r_{\nu,i} \right)\\
		+\underset{N_{K/\Q}\nu= C_K' N^d}{\sum_{\nu\in\mathcal{O}_K}}H_K(\nu)\cdot T_{n,N}(\nu).
	\end{multline*}If $N_{K/\Q}(\nu)=C_K'N^d$, by arguing as before we get that $T_{n,N}(\nu)\ll_{n,K}1$. Then the last sum is$$\ll_{n,K}\underset{N_{K/\Q}\nu= C_K' N^d}{\sum_{\nu\in\mathcal{O}_K}}1\sim\frac{C_KC_K'}{h_K}N^d=o(N^{d(n-1)})$$for $n\ge3$. Finally,\begin{multline*}
		\underset{1<|N_{K/\Q}\nu|\le C_K' N^d}{\sum_{\nu\in\mathcal{O}_K}}T_{n,N}(\nu)=\underset{2\le N_{K/\Q}\nu< C_K' N^d}{\sum_{\nu\in\mathcal{O}_K}}H_K(\nu)\cdot\prod_{i=1}^{n-1}\left(\frac{C_KC_K'}{h_K}\frac{(2N)^d}{N_{K/\Q}\nu}+r_{\nu,i} \right)\\+o(N^{d(n-1)})\\
		=\underset{2\le N_{K/\Q}\nu< C_K' N^d}{\sum_{\nu\in\mathcal{O}_K}}H_K(\nu)\cdot\left(\left(\frac{C_KC_K'}{h_K}\frac{(2N)^d}{N_{K/\Q}\nu} \right)^{n-1}+O_{n,K}(N^{d(n-3)})  \right)\\+o(N^{d(n-1)})\\
		=N^{d(n-1)}\left( \frac{2^dC_KC_K'}{h_K}\right)^{n-1} \underset{2\le N_{K/\Q}\nu< C_K' N^d}{\sum_{\nu\in\mathcal{O}_K}}H_K(\nu)\cdot\frac{1}{(N_{K/\Q}\nu)^{n-1}}\\+o(N^{d(n-1)})\\
		=N^{d(n-1)}\left( \frac{2^dC_KC_K'}{h_K}\right)^{n-1}\cdot \underset{2\le |N_{K/\Q}\nu|< C_K' N^d}{\sum_{\nu\in\mathcal{O}_K}} \frac{1}{|N_{K/\Q}\nu|^{n-1}}+o(N^{d(n-1)})\\
		=N^{d(n-1)}\left( \frac{2^dC_KC_K'}{h_K}\right)^{n-1}\cdot D_{n,K}+o(N^{d(n-1)}).
	\end{multline*}
\end{proof}Recall that $\alpha_j=\sum_{k=1}^{d}a_k^{(j)}\omega_k$ for all $j=0,\dots,n-1$. Let$$h(f):=(h_1(f),\dots,h_d(f))\in\Z^d,$$where $h_k(f)=a_{k}^{(0)}+\dots+a_k^{(n-1)}$ for all $k=1,\dots,d$. Define$$\mathscr{L}_n(N,h):=\{f\in\mathscr{P}_{n,N}(K):h(f)=h\}$$and $L_n(N,h)=|\mathscr{L}_n(N,h)|$. We have\begin{equation}
	L_n(N,h)=L_n(N,h')
\end{equation}if $h'_k=\pm h_k$ for all $k$; moreover, by a counting argument as in Lemma 3, it holds\begin{equation}
	\underset{|N_{K/\Q}\nu|=1}{\sum_{\nu\in\mathcal{O}_K}}T_{n,N}(\nu)\asymp A_{n,K}L_n(N,(1,\dots,1))
\end{equation}for some positive constant $A_{n,K}$. Note that in the left hand side, $ T_{n,N}(\nu)=0 $ for almost all $\nu\in\mathcal{O}_K^{\times}$, since $\h(\nu)$ is arbitrary large by Dirichlet's unit theorem.
\begin{lemma}
	For all $h\in\Z^d$,$$\lim_{N\rightarrow+\infty}\frac{L_n(N,h)}{N^{d(n-1)}}=k_{n,d}.$$
\end{lemma}
\begin{proof}
	By (4) we may assume that $h_k\ge0$ for all $k$. Let $f\in\mathscr{L}_n(N,0)$ and let$$f'(X)=X^n+\alpha_{n-1}X^{n-1}+\dots+\alpha_0+\sum_{k=1}^{d}h_k\omega_k.$$Then $f'\in\mathscr{L}_n(N+\underset{k}{\max}\ h_k,h)$. This implies$$L_n(N,0)\le L_n(N+\underset{k}{\max}\ h_k,h).$$Let now $f\in\mathscr{L}_n(N,h)$ and let$$f'(X)=X^n+\alpha_{n-1}X^{n-1}+\dots+\alpha_0-\sum_{k=1}^{d}h_k\omega_k.$$We have$$L_n(N,h)\le L_n(N+\underset{k}{\max}\ h_k,0).$$It follows that$$\frac{L_n(N-\max_k h_k,0)}{L_n(N,0)}\le\frac{L_n(N,h)}{L_n(N,0)}\le\frac{L_n(N+\max_k h_k,0)}{L_n(N,0)}.$$
	In particular$$L_n(N,h)\sim L_n(N,0)$$for all $h$, as $N\rightarrow+\infty$.\\
	Our claim is therefore$$\lim_{N\rightarrow+\infty}\frac{L_n(N,0)}{N^{d(n-1)}}=k_{n,d}.$$Let $E_{nd}$ be the $nd$-dimensional Euclidean space with coordinates\\ $x_1^{(0)},\dots,x_1^{(n-1)},\dots,x_d^{(0)},\dots,x_d^{(n-1)}$. Let $\Lambda_{nd}$ be the lattice of integral points in $E_{nd}$. Then $L_n(N,0)$ corresponds to the number of integral points of $\Lambda_{nd}$ which lie inside the cube$$C_N:\ |x_k(j)|\le N\ \ \ \forall j=0,\dots,n-1,\ \ \ \forall k=1,\dots,d$$and the hyperplanes$$H_k:\ x_k^{(0)}+\dots,x_k^{(n-1)}=0\ \ \ \forall k=1,\dots,d.$$That is,$$L_n(N,0)=|\Lambda_{nd}\cap C_N\cap H|,$$where $H=H_1\cap\dots\cap H_d$. $H$ is a $d(n-1)$-dimensional space; we indentify it with $E_{d(n-1)}$ with coordinates $x_k^{(0)},\dots,x_k^{(n-2)}$ for all $k=1,\dots,d$. Also,\begin{align*}
		C_N\cap H:\ |x_k^{(j)}|\le N, &&\Big|\sum_{j=0}^{n-2}y_k^{(j)}\Big|\le N
	\end{align*}for all $j=0,\dots,n-2$ and $k=1,\dots,d$.\\
	But $$\lim_{N\rightarrow+\infty}\frac{|\Lambda_{nd}\cap C_N\cap H|}{N^{d(n-1)}}=vol(R),$$where $R$ is the region obtained transforming $C_N\cap H$ by the substituition $x_k^{(j)}=Ny_k^{(j)}$ for all $j,k$. We conclude that\begin{multline*}
		\lim_{N\rightarrow+\infty}\frac{L_n(N,0)}{N^{d(n-1)}}=vol(R)=\int_{R}\dots\int dy_1^{(0)}\dots dy_1^{(n-2)}\dots dy_d^{(0)}\dots dy_d^{(n-2)}\\=k_{n,d}.
	\end{multline*}
	
\end{proof}Lemma 4 and (5) yield the following.
\begin{coroll}
	$$\lim_{N\rightarrow+\infty}\underset{|N_{K/\Q}\nu|=1}{\sum_{\nu\in\mathcal{O}_K}}\frac{T_{n,N}(\nu)}{N^{d(n-1)}}=A_{n,K}k_{n,d}.$$
\end{coroll}

\begin{proof}
	(Theorem 3) Let $n\ge3$; write\begin{align*}
		\underset{|N_{K/\Q}\nu|\le C_K' N^d}{\sum_{\nu\in\mathcal{O}_K}}T_{n,N}(\nu)&=\underset{1<|N_{K/\Q}\nu|\le C_K' N^d}{\sum_{\nu\in\mathcal{O}_K}}T_{n,N}(\nu)+\underset{|N_{K/\Q}\nu|=1 }{\sum_{\nu\in\mathcal{O}_K}}T_{n,N}(\nu)+T_{n,N}(0).
	\end{align*}Now, $T_{n,N}(0)\sim (2N)^{d(n-1)}$; by Lemma 3 and Corollary 2$$\lim_{N\rightarrow+\infty}\underset{|N_{K/\Q}\nu|\le C_K' N^d}{\sum_{\nu\in\mathcal{O}_K}}\frac{T_{n,N}(\nu)}{N^{d(n-1)}}= 2^{d(n-1)}D_{n,K}\cdot\left( \frac{C_KC_K'}{h_K}\right)^{n-1} +A_{n,K}k_{n,d}+2^{d(n-1)}.$$The theorem follows by Lemma 2. 
\end{proof}
\subsection{Proof of Theorem 1, part 2}
\subsubsection{Large sieve inequality for number fields}

Let $\alpha\in\Q^n/\Z^n$ and let $c(a)\in\Co$ for all $a$ lattice vector in $\Z^n$. Define$$S(\alpha)=\sum_{H(a)\le N}c(a)e(a\cdot\alpha),$$where the sum runs over $a\in\Z^n$ of \textit{height} $H(a)\le N$, where $H(a)$ is the maximum of the absolute values of the components of $a$. We use the standard notation $e(x)=e^{2\pi ix}$. Let $\ord(\alpha)=\min\{m\in\N:m\alpha\in\Z^n\}$. The following is the multidimensional analogue of the Bombieri-Davenport inequality:$$\sum_{\ord(\alpha)\le x}|S(\alpha)|^2\ll_n(N^s+x^{2s})\sum_{H(a)\le N}|c(a)|^2.$$We want a similar estimate for algebraic number fields. Specifically, let $\mathfrak{a}$ be an integral ideal of $K$, and let $\sigma$ be an additive character of $\mathcal{O}_K^n$ mod $\mathfrak{a}$. We call $\sigma$ \textit{proper} if it is not a character mod $\mathfrak{b}$ for any $\mathfrak{b}|\mathfrak{a}$. Let $c(\xi)\in\Co$ for all $\xi=(\xi_1,\dots,\xi_n)\in\mathcal{O}_K^n$. As before, define$$S(\sigma)=\sum_{H(\xi)\le N}c(\xi)\sigma(\xi),$$where $H(\xi)=\max_{i=1}^{n}\h(\xi_i)$.
\begin{prop}
	One has$$\sum_{N_{K/\Q}\mathfrak{a}\le x}\sum_{\sigma}|S(\sigma)|^2\ll_{n,K}(N^{nd}+c_Kx^{2n})\sum_{H(\xi)\le N}|c(\xi)|^2$$for some constant $c_K$, where the second sum is over the proper additive characters $\modulus \mathfrak{a}$.
\end{prop}
A more precise statement of this result can be found in \cite{Hu}, Theorem 2 for the 1-dimensional case. Proposition 2 is the multidimensional analogue which can be achieved as for the case $K=\Q$; for further details see \cite{Hu}, again.\\

Let now $\wp$ be a prime ideal of $\mathcal{O}_K$ and let $\Omega(\wp)$ be a subset of $\mathcal{O}_K^n/\wp\mathcal{O}_K^n$, whom order is $\nu(\wp)$, say. For each $\xi\in\mathcal{O}_K^n$, set$$P(\xi,x)=|\{\wp\in\mathcal{O}_K: N_{K/\Q}\wp\le x,\ \xi\mbox{ mod }\wp\in\Omega(\wp)\}|$$and$$P(x)=\sum_{N_{K/\Q}\wp\le x}\frac{\nu(\wp)}{q_{\wp}^n},$$where $q_{\wp}=N_{K/\Q}\wp$.

The next results are classical applications of Proposition 2. We include the proofs for completeness.
\begin{lemma}
	If $N\gg_K x^{2/d}$, then$$\sum_{H(\xi)\le N}(P(\xi,x)-P(x))\ll_{n,K}N^{nd}P(x).$$
\end{lemma}
\begin{proof}
	Let $\varphi_\wp$ be the characteristic function of the set $\Omega(\wp)$, that is$$\varphi_\wp(\xi)=\begin{cases}
		1&\mbox{if }\xi\mbox{ mod }\wp\in\Omega(\wp)\\
		0&\mbox{otherwise}.
	\end{cases}$$It is periodic function mod $\wp$. Its Fourier transform is$$\widehat{\varphi_\wp}(\sigma)=\frac{1}{q_\wp^n}\sum_{\xi\tiny\mbox{ mod }\wp}\varphi_\wp(\xi)\overline{\sigma(\xi)}.$$By the inversion formula we get$$\varphi_\wp(\xi)=\sum_{\sigma\tiny\mbox{ mod }\wp}\widehat{\varphi_\wp}(\sigma)\sigma(\xi),$$where the sum is over the characters mod $\wp$. In particular\begin{equation}
		\widehat{\varphi_\wp}(1)=\frac{\nu(\wp)}{q_\wp^n}
	\end{equation}and, by the orthogonality relations,\begin{equation}
		\sum_{\sigma\tiny\mbox{ mod }\wp}|\widehat{\varphi_\wp}(\sigma)|^2=\frac{\nu(\wp)}{q_\wp^n}
	\end{equation}From (6), we can write$$P(\xi,x)=\sum_{N_{K/\Q}\wp\le x}\varphi_\wp(\xi)=P(x)+R(\xi,x),$$where$$R(\xi,x)=\sum_{N_{K/\Q}\le x}\underset{\sigma\neq1}{\sum_{\sigma\tiny\mbox{ mod }\wp}}\widehat{\varphi_\wp}(\sigma)\sigma(\xi).$$By the Cauchy-Schwartz inequality one has\begin{align*}
		\sum_{H(\xi)\le N}(R(\xi,x))^2&=\sum_{N_{K/\Q}\wp\le x}\underset{\sigma\neq1}{\sum_{\sigma\tiny\mbox{ mod }\wp}}\widehat{\varphi_\wp}(\sigma)\sum_{H(\xi)\le N}R(\xi,x)\sigma(\xi)\\
		&\le\left( \sum_{N_{K/\Q}\wp\le x}\underset{\sigma\neq1}{\sum_{\sigma\tiny\mbox{ mod }\wp}}|\widehat{\varphi_\wp}(\sigma)|^2\right)^{1/2}\cdot\left( \sum_{\sigma}|S(\sigma)|^2\right)^{1/2}, 
	\end{align*}where the last sum is over $\sigma\neq 1$ mod $\wp$ for some $\wp$ of norm $\le x$, and$$S(\sigma)=\sum_{H(\xi)\le N}R(\xi,x)\sigma(\xi).$$From (7) and Proposition 2 we get$$\sum_{H(\xi)\le N}(R(\xi,x))^2\ll\left(P(x)(N^{nd}+c_Kx^{2n})\sum_{H(\xi)\le N}|R(\xi,x)|^2\right)^{1/2},$$which, for $N\gg x^{2/d}$, implies the lemma.
\end{proof}	
Define, for a collection of subsets $\Omega(\wp)$ for each prime $\wp$, $$E(N)=|\{\xi\in\mathcal{O}_K^n:H(\xi)\le N,\ \xi\mbox{ mod }\wp\notin\Omega(\wp)\ \forall\wp\}|.$$Put$$\mathcal{S}(x)=\sum_{N_{K/\Q}\mathfrak{a}\le x}\mu^2(\mathfrak{a})\prod_{\wp|\mathfrak{a}}\frac{\nu(\wp)}{q_\wp^n-\nu(\wp)},$$where $\mu$ is the M\"{o}bius function; in particular$$\mu^2(\mathfrak{a})=\begin{cases}
	1&\mbox{if }\mathfrak{a}\mbox{ is square-free}\\
	0&\mbox{otherwise}.
\end{cases}$$	
\begin{lemma}
	If $N\gg_K x^{2/d}$, then$$E(N)\ll_{n,K}N^{nd}\mathcal{S}(x)^{-1}.$$
\end{lemma}
\begin{proof}
	Let$$c(\xi)=\begin{cases}
		1&\mbox{if }\xi\mbox{ mod }\wp\notin\Omega(\wp)\ \forall\wp\\
		0&\mbox{otherwise}
	\end{cases}$$for all $\xi\in\mathcal{O}_K^n$.
	Note that$$E(N)=\sum_{H(\xi)\le N}|c(\xi)|^2=S(1).$$If we show that\begin{equation}
		\sum_{\sigma\tiny\mbox{ mod }\mathfrak{a}}|S(\sigma)|^2\ge|S(1)|^2\prod_{\wp|\mathfrak{a}}\frac{\nu(\wp)}{q_\wp^n-\nu(\wp)}
	\end{equation}for all square-free $\mathfrak{a}$, then we have, by Proposition 2, for $N\gg_K x^{2/d}$,\begin{align*}
		E(N)^2\mathcal{S}(x)&\le\sum_{N_{K/\Q}\mathfrak{a}\le x}\mu^2(\mathfrak{a})\sum_{\sigma\tiny\mbox{ mod }\mathfrak{a}}|S(\sigma)|^2\\
		&\ll(N^{nd}+c_Kx^{2n})\sum_{H(\xi)\le N}|c(\xi)|^2\\
		&\ll N^{nd}E(N),
	\end{align*}and the lemma follows.\\
	\underline{Proof of (8)}: for every prime $\wp$, by orthogonality we have\begin{equation}
		\sum_{\sigma\tiny\mbox{ mod }\wp}|S(\sigma)|^2=q_\wp^n\sum_{\zeta\in\mathcal{O}_K^n/\wp\mathcal{O}_K^n}|S(\zeta,\wp)|^2-|S(1)|^2
	\end{equation}where $S(\zeta,\wp)=\underset{\xi\tiny\mbox{ mod }\wp\in\zeta}{\sum_{\xi}}c(\xi)$. By the Cauchy-Schwartz inequality,\begin{equation}
		|S(1)|^2=\Big| \sum_{\zeta}S(\zeta,\wp)\Big|  ^2\le(q_\wp-\nu(\wp))\sum_{\zeta}|S(\zeta,\wp)|^2
	\end{equation}since $S(\zeta,\wp)=0$ for all $\zeta\in\Omega(\wp)$. Equations (9) and (10) imply (8) for the case $\mathfrak{a}=\wp$ prime ideal.
	
	More generally, if $\sigma_1$ is a character mod $\wp$, one has$$\sum_{\sigma\tiny\mbox{ mod }\wp}|S(\sigma\cdot\sigma_1)|^2\ge|S(\sigma_1)|^2\frac{\nu(\wp)}{q_\wp^n-\nu(\wp)}$$by replacing $c(\xi)$ with $c(\xi)\sigma_1(\xi)$.
	
	Let now $\mathfrak{a}$ be square-free. By the unique factorization of ideals we can write $\mathfrak{a}=\wp\mathfrak{b}$ for some prime ideal $\wp$ and for some square-free ideal $\mathfrak{b}$, with $\wp\nmid\mathfrak{b}$. The chinese remainder theorem gives,\begin{align*}
		\sum_{\sigma\tiny\mbox{ mod }\mathfrak{a}}|S(\sigma)|^2&=\sum_{\sigma\tiny\mbox{ mod }\wp}\ \sum_{\sigma_1\tiny\mbox{ mod }\mathfrak{b}}|S(\sigma\cdot\sigma_1)|^2\\
		&\ge \frac{\nu(\wp)}{q_\wp^n-\nu(\wp)}\sum_{\sigma_1\tiny\mbox{ mod }\mathfrak{b}}|S(\sigma_1)|^2.
	\end{align*}We conclude by induction on the number of prime factors of $\mathfrak{a}$. 
\end{proof}

\subsubsection{Sieving polynomials in $\mathscr{P}_{n,N}$}

Let $f\in\mathscr{P}_{n,N}$, $f(X)=X^n+\alpha_{n-1}X^{n-1}+\dots+\alpha_0$. We identify $f$ with the lattice vector $\xi=\xi_f=(\alpha_{n-1},\dots,\alpha_0)$ formed by its coefficients, so that $H(\xi_f)=\h(f)$. Similarly, polynomials mod $\wp$ are identified with lattice vectors mod $\wp$.
\begin{prop}
	Let $r$ be a splitting type. If $N\gg_K x^{2/d}$, then$$\sum_{f\in\mathscr{P}_{n,N}}(\pi_{f,r}(x)-\delta(r)\pi_K(x))^2\ll_{n,K}N^{nd}\pi_K(x).$$
\end{prop}
\begin{proof}
	For every prime $\wp$ of $K$ of norm $q_{\wp}$, let
	\begin{align*}
		X_{n,r,\wp}=\Big\{&\Big( \prod_{i=1}^{r_1}g_i^{(1)}\Big) \dots\Big(  \prod_{i=1}^{r_n}g_i^{(n)}\Big):g_i^{(j)}\in\mathbb{F}_{q_{\wp}}[X]\ \mbox{irreducile, monic},\\
		&\deg(g_i^{(j)})=j,\ g_i^{(j)}\neq g_k^{(j)}\mbox{ if }i\neq k\Big\}.
	\end{align*}Namely, $\Omega(\wp)=\Omega_r(\wp):=X_{n,r,\wp}$ is the set of polynomials of (square-free) splitting type $r$ in the finite field $\mathbb{F}_{q_{\wp}}$. It is not hard to see that$$\nu_r(\wp)=|X_{n,r,\wp}|=\delta(r)q_\wp^n+O(q_\wp^{n-1})$$(see for instance \cite{Vi}, Proposition 1). Therefore $$P(x)=\sum_{N_{k/\Q}\wp\le x}\frac{\nu_r(\wp)}{q_\wp^n}=\delta(r)\pi_K(x)+O(\log\log x),$$and $\pi_{f,r}(x)=P(\xi_f,x)$. The proposition thus follows by Lemma 5.
\end{proof}

For any $f\in\mathcal{O}_K[X]$, let$$\pi_f(x):=\underset{\wp\tiny\mbox{ unramified}}{\sum_{q_\wp\le x}}|\{\alpha\in\mathcal{O}_K:f(\alpha)\equiv0\mbox{ mod }\wp\}|.$$Observe that if $f$ is irreducible and $f(\alpha)\equiv0$ mod $\wp$ unramified, then there is a prime $\mathfrak{P}|\wp$ in the field $K(\alpha)$ so that $[\mathcal{O}_{K(\alpha)}/\mathfrak{P}:\mathcal{O}_K/\wp]=1$, i.e. $N_{K/\Q}\mathfrak{P}=q_\wp$. Therefore $\pi_f(x)$ corresponds to the prime ideal counting function $\pi_{K(\alpha)}(x)$, and the asymptotic$$\pi_f(x)\sim \pi_K(x),$$as $x\rightarrow+\infty$, holds by the Prime Ideal Theorem.
\begin{coroll}
	If $N\gg_K x^{2/d}$, then$$\sum_{f\in\mathscr{P}_{n,N}}(\pi_f(x)-\pi_K(x))^2\ll_{n,K} N^{nd}\pi_K(x).$$
\end{coroll}
\begin{proof}
	Write\begin{align*}
		\pi_f(x)&=\sum_r r_1\pi_{f,r}(x)\\
		&=\Big( \sum_r r_1\delta(r)\Big) \pi_K(x)+\sum_r r_1(\pi_{f,r}(x)-\delta(r)\pi_K(x)).
	\end{align*}In order to compute $\sum_r r_1\delta(r)$ we consider the generating function$$\sum_{n\ge0}\Big(\underset{\sum ir_i=n}{\sum_{r_1,\dots,r_n\ge0}}r_1\delta(r)\Big) X_1^{r_1}X^{n-r_1}=\sum_{r_1,\dots,r_n\ge0}\frac{X_1^{r_1}}{r_1!}\prod_{i=2}^{n}\frac{1}{i^{r_i}r_i!}X^{2r_2}X^{3r_3}\dots.$$We have that $\sum_r r_1\delta(r)$ corresponds to the coefficients of $X^{n-1}$ of\begin{align*}
		\frac{\partial}{\partial X_1}\Big|_{X_1=X}\exp {\Big( X_1+\sum_{n\ge 2}\frac{X^n}{n}\Big) } &=\frac{\partial}{\partial X_1}\Big|_{X_1=X}\exp{\Big( X_1+\int_{0}^{X}\frac{dt}{1-t}-X\Big) }\\
		&=\frac{1}{1-X}\\
		&=1+X+X^2+\dots
	\end{align*}which is 1. It turns out, by the Cauchy-Schwartz inequality, that$$(\pi_f(x)-\pi_K(x))^2\ll\sum_{r}r_1^2(\pi_{f,r}(x)-\delta(r)\pi_K(x))^2,$$and we apply Proposition 3 for each $r$.
\end{proof}
Fix a splitting type $r$, and let$$E_r(N)=|\{f\in\mathscr{P}_{n,N}:f\mbox{ has splitting type }r\mbox{ mod }\wp\mbox{ for no prime }\wp\}|.$$\begin{coroll}
	One has$$E_r(N)\ll_{n,K}N^{d(n-1/2)}\log N.$$
\end{coroll}
\begin{proof}
	For $f\in E_r(N)$, $\pi_{f,r}(x)=0$, so by\begin{multline*}
		\sum_{f\in\mathscr{P}_{n,N}}(\pi_{f,r}(x)-\delta(r)\pi_K(x))^2\\
		=\underset{\pi_{f,r}(x)\neq0}{\sum_{f\in\mathscr{P}_{n,N}}}(\pi_{f,r}(x)-\delta(r)\pi_K(x))^2+\sum_{f\in E_r(N)}(\delta(r)\pi_K(x))^2\\
		=\underset{\pi_{f,r}(x)\neq0}{\sum_{f\in\mathscr{P}_{n,N}}}(\pi_{f,r}(x)-\delta(r)\pi_K(x))^2+E_r(N)(\delta(r)\pi_K(x))^2
	\end{multline*}we get$$E_r(N)\asymp\underset{\pi_{f,r}(x)\neq0}{\sum_{f\in\mathscr{P}_{n,N}}}(\pi_{f,r}(x)-\delta(r)\pi_K(x))^2(\pi_K(x))^{-2}\ll_{n,K}N^{nd}(\pi_K(x))^{-1}$$for $N\gg_K x^{2/d}$. Pick $x\asymp N^{d/2}$ and conclude by the prime ideal theorem.
\end{proof}
In order to improve the exponent of $\log N$ we apply Lemma 6 to $E_R(N)$, where $R$ is a nonempty set of splitting types and $$E_R(N)=|\{f\in\mathscr{P}_{n,N}:f\mbox{ has splitting type in }R\mbox{ mod }\wp\mbox{ for no prime }\wp\}|.$$Put $\delta(R)=\sum_{r\in R}\delta(r)$.\begin{prop}
	For any $\delta<\delta(R)$, one has$$E_R(N)\ll_{n,K}N^{d(n-1/2)}(\log N)^{1-\frac{\delta}{1-\delta}}.$$
\end{prop}
\begin{proof}
	Let $\Omega_R(\wp)=\bigcup_{r\in R}X_{n,r,\wp}$. Its order is$$\nu_{R}(\wp)=\delta(R)q_{\wp}^n+O(q_{\wp}^{n-1}).$$If the norm of $\wp$ is large enough, $q_\wp\ge t$, say, we have $$\nu_R(\wp)\ge\delta q_{\wp}^n.$$If $N\gg_{n,K} x^{2/d}$, Lemma 6 gives$$E_R(N)\ll_{n,K}N^{nd}\mathcal{S}_R(x)^{-1},$$where$$\mathcal{S}_R(x)=\sum_{N_{K/\Q}\mathfrak{a}\le x}\mu^2(\mathfrak{a})\prod_{\wp|\mathfrak{a}}\frac{\nu_R(\wp)}{q_\wp^n-\nu_R(\wp)}.$$For $q_{\wp}\ge t$,$$\frac{\nu_R(\wp)}{q_\wp^n-\nu_R(\wp)}=\left( \frac{q_\wp^n}{\nu_R(\wp)}-1\right)^{-1}\ge\frac{\delta}{1-\delta} ,$$so$$\mathcal{S}_R(x)\ge\underset{\wp|\mathfrak{a} \Rightarrow  q_{\wp}\ge t}{\sum_{N_{K/\Q}\mathfrak{a}\le x}}\mu^2(\mathfrak{a})\prod_{\wp|\mathfrak{a}}\frac{\delta}{1-\delta}=\underset{\wp|\mathfrak{a} \Rightarrow q_{\wp}\ge t}{\sum_{N_{K/\Q}\mathfrak{a}\le x}}\mu^2(\mathfrak{a})\left(\frac{\delta}{1-\delta} \right)^{\omega(\mathfrak{a})}. $$We hence need a lower bound for the sum$$\mathcal{S}_{\gamma, t}(x)=\underset{\wp|\mathfrak{a} \Rightarrow q_{\wp}\ge t}{\underset{\mathfrak{a}\tiny\mbox{ square-free}}{\sum_{N_{K/\Q}\mathfrak{a}\le x}}}\gamma^{\omega(\mathfrak{a})},$$where $\gamma:=\delta/(1-\delta)$. By a result of Selberg (\cite{Sel}, Theorem 2), we get$$\mathcal{S}_{\gamma, t}(x)\asymp\frac{1}{\Gamma(\gamma)}\prod_{q_\wp<t}\left( 1-\frac{1}{q_\wp}\right)^\gamma \prod_{q_\wp\ge t}\left( 1+\frac{\gamma}{q_\wp}\right) \left( 1-\frac{1}{q_\wp}\right)^\gamma x(\log x)^{\gamma-1}.$$Putting $x\asymp_{n,K}N^{d/2}$, the claim follows.
\end{proof}
As we said in the introduction, if $f$ mod $\wp$ has splitting type $r$, for some unramified $\wp$, then the Frobenius at $\wp$ has cycle structure $r$. If $G\subset S_n$ is a proper subgroup, it's a standard fact that the conjugates of $G$ do not cover $S_n$; thus $f$ cannot have all the splitting types. It follows that$$|\mathscr{P}_{n,N}\setminus\mathscr{P}_{n,N}^0|\le\sum_{r}E_r(N).$$We conclude, by Corollary 2, that$$|\mathscr{P}_{n,N}\setminus\mathscr{P}_{n,N}^0|\ll_{n,K}N^{d(n-1/2)}\log N.$$
We are now ready to prove Theorem 1, part (2). We use the following lemma to improve the exponent of $\log N$.
\begin{lemma}
	If $G$ is a transitive subgroup of $S_n$, contains a transposition and contains a $p$-cycle for some $p>n/2$, then $G=S_n$.
\end{lemma}
\begin{proof}
	See \cite{Gal}, page 98.
\end{proof}
Let\begin{align*}
	T&=\{r:r_2=1,\ r_4=r_6=\dots=0\},\\
	P&=\{r:r_p=1\mbox{ for some }p>n/2\}.
\end{align*}By using the above correspondence, we can view $T$ as the set of elements of $S_n$ among whose cycles there is just one transposition and no other cycles of even length. Analogously, $P$ is the set of elements of order divisible by some prime $p>n/2$. By Lemma 7, we have the inequality\begin{equation}
	|\mathscr{P}_{n,N}\setminus\mathscr{P}_{n,N}^0|\le\rho(n,N)+E_T(N)+E_P(N)
\end{equation}We can estimate $\rho(n,N)$ thanks to Proposition 1. For the other summands, we compute $\delta(T)$ and $\delta(P)$ in order to apply Proposition 4.\begin{enumerate}
	\item [$\bullet$] Write$$\delta(T)=\frac{1}{2}\underset{\sum ir_i=n-2}{\sum_{r_3,r_5,\dots}}\ \underset{\tiny\mbox{odd}}{\prod_{i\ge3}}\ \frac{1}{i^{r_i}r_i!}.$$The generating function is$$\frac{1}{2}\sum_{r_3,r_5\dots}\underset{\tiny\mbox{odd}}{\prod_{i\ge3}}\ \frac{1}{i^{r_i}r_i!}X^{2+\underset{\tiny\mbox{odd}}{\sum_{i\ge3}ir_i}}=\frac{X^2}{2}\exp{\left( \sum_{n\ge0}\frac{X^{2n+1}}{2n+1}\right) }.$$Therefore $\delta(T)$ is half the coefficient of $X^{n-2}$ of\begin{align*}
		\exp{\left( \sum_{n\ge0}\frac{X^{2n+1}}{2n+1}\right)}&=\exp\left(\int_{0}^{X}\frac{dt}{1-t^2} \right)\\
		&=\exp{\left( \frac{1}{2}\int_{0}^{X}\frac{dt}{1-t}\right) }\exp{\left( \frac{1}{2}\int_{0}^{X}\frac{dt}{1+t}\right) } \\
		&=\left( \frac{1+X}{1-X}\right)^{1/2}\\
		&=(1+X)(1-X^2)^{-1/2}\\
		&=(1+X)\frac{\partial}{\partial X}(\arcsin X)\\
		&= (1+X)\frac{\partial}{\partial X}\left( \sum_{k\ge0}\frac{1}{2^{2k}}\binom{2k}{k}X^{2k+1}\right)\\
		&= (1+X)\frac{\partial}{\partial X}\left(  \sum_{k\ge0}\frac{(2k)!}{(2^kk!)^2}X^{2k}\right) .
	\end{align*}It turns out that$$\delta(T)=\frac{(n-j)!}{2^{n-j+1}\big(\frac{n-j}{2} \big)!^2 },$$where $j=2$ if $n$ is even and $j=3$ if $n$ is odd. By Stirling's approximation $n!\sim\big(\frac{n}{e}\big)^n\sqrt{2\pi n}$ we get, for instance, when $n$ is even,\begin{align*}
		\delta(T)&\sim\frac{\big(\frac{n-2}{e}\big)^{n-2}\sqrt{2\pi (n-2)}}{2^{n-2}\big(\big(\frac{n-2}{2}\big)^{\frac{n-2}{2}}\sqrt{\pi(n-2)}\big)^{2}}\\
		&\sim\frac{1}{\sqrt{2\pi n}}.
	\end{align*}The case $n$ odd is analogous.
	\item [$\bullet$] Write$$
	\delta(P)=\sum_{n/2<p\le n}\frac{1}{p}\underset{\sum_{i\neq p}ir_i=n-p}{\sum_{r_i,i\neq p}}\prod_{i\neq p}\frac{1}{i^{r_i}r_i!}.
	$$The generating function of the last sum above for a fixed prime $n/2<p\le n$ is$$\sum_{r_i,i\neq p}\ \prod_{i\neq p}\frac{1}{i^{r_i}r_i!}X^{p+\sum_{i\neq p}ir_i}=X^p\exp\left(\sum_{n\ge1,\ n\neq p}\frac{X^{n}}{n} \right). $$The coefficient of $X^{n-p}$ of $\exp\left(\sum_{n\ge1,\ n\neq p}\frac{X^{n}}{n} \right)$ is precisely our sum, which is therefore 1, since\begin{align*}
		\exp\left(\sum_{n\ge1,\ n\neq p}\frac{X^{n}}{n} \right)&=\exp{\left( \int_{0}^{X}\frac{dt}{1-t}\right) }\\
		&=\exp(-\log(1-X))\\
		&=1+X+X^2+\dots.
	\end{align*}By the classical Martens' estimate, we conclude that$$\delta(P)=\sum_{ n/2< p\le n}\frac{1}{p}\sim\frac{\log 2}{\log n}.$$
\end{enumerate}
By (11), Lemma 7 and Proposition 1 it follows$$|\mathscr{P}_{n,N}\setminus\mathscr{P}_{n,N}^0|\ll_{n,K}N^{d(n-1/2)}(\log N)^{1-\gamma_n},$$where $\gamma_n\sim(2\pi n)^{-1/2}$, that is, part (2) of Theorem 1.

\subsubsection*{Remarks}
Let $f\in\mathscr{P}_{n,N}$. By Proposition 3, we get in particular that for every $\varepsilon>0$,$$\pi_{f,r}(x)-\delta(r)\pi_K(x)=O\left( x^{\frac{1}{2}}\log x\right) $$as $x\rightarrow+\infty$, for all but $O_{n,K}(x^{2n}(\log x)^{-3})$ polynomials $f$ with $\h(f)\ll x^{2/d}$. Indeed, if$$E(x)=\{f\in\mathscr{P}_{n,N}:|\pi_{f,r}(x)-\delta(r)\pi_K(x)|>x^{\frac{1}{2}}\log x\}$$denotes the exceptional set, one has\begin{align*}
	x(\log x)^2|E(x)|&\ll\sum_{f\in\mathscr{P}_{n,N}}|\pi_{f,r}(x)-\delta(r)\pi_K(x)|^2\\
	&\ll N^{nd}\pi_K(x),
\end{align*}if $N\gg x^{2/d}$. Hence$$|E(x)|\ll\frac{N^{nd}}{x(\log x)^2}\frac{x}{\log x}\ll \frac{x^{2n}}{(\log x)^3}$$by setting $N\asymp x^{2/d}$ and by letting $x\rightarrow+\infty$.

This sharper form$$\pi_{f,r}(x)-\delta(r)\pi_K(x)=O( x^{\frac{1}{2}+\varepsilon})$$ holds for all irreducible $f$ by assuming the Artin's conjecture for the splitting field of $f$.\\

\subsection{Proof of Theorem 1, part 3}
In order to conclude, it remains to show (1) for $G$ primitive subgroup and for $G$ transitive but imprimitive subgroup.

\subsubsection{Case 1: $G$ imprimitive}
The irreducible polynomials $f\in\mathscr{P}_{n,N}(K)$ having such $G$ as Galois group are those whose associated field $L_f=K[X]/(f)=K(\alpha)$ ($\alpha$ is any root of $f$) has a nontrivial subfield over $K$.\\
Note that for any proper divisor $e$ of $n$,\begin{multline*}
	|\{f\in\mathscr{P}_{n,N}(K):f\mbox{ irreducible, }K(\alpha)/K\mbox{ has a subfield of degree }e\}|\\
	\le|\{\beta\in\Qal:[K(\beta):K]=n,\ K(\beta)/K\mbox{ has a subfield of degree }e,\ H_K(\beta)\ll N^{1/n}\}|\\
	\le|\{\theta\in\Qal:[\Q(\theta):\Q]=nd,\ \Q(\theta)/\Q\mbox{ has a subfield of degree }ed,\ H(\theta)\ll_K N^{1/n}\}|
\end{multline*}
$$\begin{tikzpicture}
	\matrix (m) [matrix of math nodes,row sep=1em,column sep=0.00001em,minimum width=2em]
	{K(\beta)&=\Q(\theta)\\
		K&\\
		\Q& \\};
	\path[-]
	(m-1-1) edge node [left] {$\scriptstyle{n!}$} (m-2-1)
	(m-2-1) edge node [left] {$\scriptstyle{d}$} (m-3-1)
	(m-1-2) edge node [right] {$\scriptstyle{nd}$} (m-3-1)
	
	;
\end{tikzpicture}$$	
We recall that for a monic polynomial $f\in\Co[X]$, the \textit{Mahler measure} of $f$ is$$M(f)=\sum_{f(\theta)=0}\max\{1,|\theta|\}.$$For any $x\in\Qal$ and $L/K$ number field containing $x$, we define the multiplicative Weil height of $x$ over $K$ as$$H_K(x)=\prod_{\nu\in M_L}\max\{1,|x|_{\nu}\}^{[L_\nu:K_\nu]/[L:K]},$$where $\nu$ runs over all the places of $L$ (note that $H_K(x)$ does not depend on the choice of $L$). For $K=\Q$, $H_{\Q}=H$ is the usual \textit{multiplicative Weil height}. If $\alpha$ is an algebraic number of degree $n$ over $K$ and $f$ is its minimal polynomial over $K$, then$$M(f)=H_K(\alpha)^n.$$Mahler showed that $M(f)$ and $\h(f)$ are commensurate in the sense that$$\h(f)\ll M(f)\ll \h(f).$$In particular $H_K(\alpha)\ll N^{1/n}$, which explains the first inequality above.\\
For the second one, note that $H(\theta)\le H_K(\theta)$ for all $\theta\in\Qal$. Moreover, if we fix a primitive element $\gamma\in K$ so that $K=\Q(\gamma)$, we have that $K(\beta)=\Q(\theta)$, where $\theta=\beta+q\gamma$ for all but finitely many $q\in\Q$. Since$$H_K(\beta+q\gamma)\le 2H_K(\beta)H_K(q\gamma),$$it follows that $H(\theta)\ll_{K}N^{1/n}.$\\

An upper bound for the set\begin{align*}
	Z(ed,n/e,c_KN^{1/n}):=\{&\theta\in\Qal:[\Q(\theta):\Q]=nd,\\\
	& \Q(\theta)/\Q\mbox{ has a subfield of degree }ed,\ H(\theta)\le c_K N^{1/n}\}
\end{align*}is given by Widmer (\cite{Wi}, Theorem 1.1.), namely$$Z(ed,n/e,c_KN^{1/n})\ll_{n,K}N^{d\big(\frac{n}{e}+ed\big)}.$$

Finally\begin{multline*}
	|\{f\in\mathscr{P}_{n,N}(K):f\mbox{ irreducible, }K(\alpha)/K\mbox{ has a non trivial subfield}\}|\\
	\ll_{n,K}\underset{e|n}{\underset{1<e<n}{\max}}N^{d\big(\frac{n}{e}+ed\big)}\le N^{d\big(\frac{n}{2}+2d\big)},
\end{multline*}because the function $d\big(\frac{n}{x}+xd\big)$ assumes the maximum in $x=2$ for $x\in[2,n/2]$.
Since it is known, for instance from Kuba \cite{Ku}, that$$\lim_{N\rightarrow+\infty}\mathbb{P}(f\mbox{ irreducible})=1$$with error term $O(N^{-d})$, we get$$\underset{\tiny\mbox{imprimitive}}{\sum_{G\subset S_n}}N_n(N,G)\ll_{n,K} N^{d\big(\frac{n}{2}+2d\big)}\ll N^{d(n-1)},$$as long as $n\ge 2(2d+1)$.

\subsubsection{Case 2: $G$ primitive}

We need the following result which generalizes a result of Lemke Oliver and Thorne (\cite{LT}, Theorem 1.3).\\

Let $G$ be a transitive subgroup of $S_n$. Any $f\in\mathscr{N}_n(N,G)$ (which can be assumed to be irreducible) cuts out a field $L_f=K[X]/(f)$ whose normal closure $K_f/K$ has Galois group $G$ with discriminant of norm $$|N_{K/\Q}\mathfrak{D}_{L_f/K}|\ll_{n,K}N^{d(2n-2)}.$$

Let $L/K$ be an extension of degree $n$. Define$$M_L(N;K)=M_L(N)=|\{f\in\mathscr{P}_{n,N}:L_f\simeq L\}|.$$

By a theorem of Schmidt \cite{Sc}, the number of field extensions $L/K$ of degree $n$ with $|N_{K/\Q}\mathfrak{D}_{L/K}|\le X$ is $O_{n,K}(X^{(n+2)/4})$. Denote by$$\mathscr{F}_n(X,G;K)=\mathscr{F}_n(X,G)=\{L/K:[L:K]=n,\ G_{\widetilde{L}/K}\cong G,\ |N_{K/\Q}\mathfrak{D}_{L/K}|\le X\},$$where $\widetilde{L}$ is the Galois closure of $L$ over $K$.
\begin{thm}
	For any $G\subseteq S_n$ transitive subgroup, one has$$N_n(N,G)\ll_{n,K}N^{d\big(1+\frac{(2n-2)(n+2)}{4}\big)}\cdot(\log N)^{nd-1}.$$If moreover $G$ is primitive,$$N_n(N,G)\ll_{n,K}N^{d\big(1+\frac{(2n-2)(n+2)}{4}\big)-\frac{2}{n}}\cdot(\log N)^{nd-1}.$$
\end{thm}

\begin{proof}
	Since the discriminant of $f\in\mathscr{P}_{n,N}$ satisfies $N_{K/\Q}d_f\ll_{n,K}N^{d(n-2)}$, we can write\begin{equation}
		N_n(N,G)\ll_{n,K}\sum_{L\in\mathscr{F}_n(N^{d(n-2)},G)}M_L(N)
	\end{equation}Now, for any $L/K$ as above with signature $(r_1,r_2)$,\begin{align*}
		M_L(N)&\le|\{\alpha\in\mathcal{O}_L:K(\alpha)\cong L,\ H_K(\alpha)\ll_{n,K}N^{1/n}\}|\\
		&\ll_{n,K}|\Omega_{N^{1/n}}\cap(\mathcal{O}_L\setminus K)|,
	\end{align*}where for $Y\ge1$, $\Omega_Y$ is the subset of the Minkowski space $L_\infty=\R^{r_1}\times\Co^{r_2}$ of elements whom Weil height over $K$ is at most $Y$.\\
	By applying Davenport's lemma and by computing the volume of $\Omega_Y$ we achieve$$|\Omega_Y\cap\Z^n|\ll_{n,d}Y^{nd}(\log Y)^{r_1+r_2-1}.$$By Proposition 2.2 of \cite{LT},$$|\Omega_Y\cap\mathcal{O}_L|\ll_{n,d}Y^{nd}(\log Y)^{r_1+r_2-1}$$as well. In particular$$M_L(N)\ll_{n,K}N^d(\log N)^{r_1+r_2-1}.$$As in the last part of the proof of Theorem 2.1 of \cite{LT}, one gets the improvement$$M_L(N)\ll_{n,K}\frac{N^d(\log N)^{r_1+r_2-1}}{\lambda},$$where $\lambda=\{\lVert\alpha\rVert:\alpha\in\mathcal{O}_L\setminus K\}$, $\lVert\alpha\rVert$ is the largest archimedean valuation of $\alpha$.
	
	By (12) and the result of Schmidt follows the first part of the theorem.\\
	
	Let now $G$ be primitive; in particular $L/K$ has no proper subextensions. Therefore essentially as in \cite{EV}, Lemma 3.1, since if $\alpha\in\mathcal{O}_L\setminus K$ then $L=K(\alpha)$, and $\mathcal{O}_K[\alpha]$ is a subring of $\mathcal{O}_L$ which generates $\mathcal{O}_L$ as a $K$-vector space, one deduces$$\lVert\alpha\rVert\gg|N_{K/\Q}\mathfrak{D}_{L/K}|^{\frac{1}{nd(n-1)}}.$$Finally, by partial summation we conclude\begin{align*}
		N_n(N,G)&\ll_{n,K}\sum_{L\in\mathscr{F}_n(N^{d(2n-2)},G)}\frac{N^d(\log N)^{nd-1}}{|N_{K/\Q}\mathfrak{D}_{L/K}|^{\frac{1}{nd(n-1)}}}\\
		&\ll_{n,K}N^{d\big(1+\frac{(2n-2)(n+2)}{4}\big)-\frac{2}{n}}\cdot(\log N)^{nd-1}.
	\end{align*}
\end{proof}

Assume now $G$ to be primitive; for $g\in G$, $g=c_1\dots c_t$ where $c_j$ are disjoint cycles, the index of $g$ is $$\mbox{ind}(g)=n-t.$$The \textit{index} of the group $G$ is$$\mbox{ind}(G)=\underset{g\neq1}{\min_{g\in G}}\ \mbox{ind}(g).$$
\begin{prop}
	Let $f\in\mathcal{O}_K[X]$ be a monic, irreducible polynomial of degree $n$ with associated field $L_f=K[X]/(f)$. If $\ind(G_f)=k$, then the discriminant $\mathfrak{D}_{L_f/K}$ has the property $\Pm_k$: if $\wp\subseteq\mathcal{O}_K$, $\wp|\mathfrak{D}_{L_f/K}$, then $\wp^k|\mathfrak{D}_{L_f/K}$.
\end{prop}	
\begin{proof}
	The Galois group $G_f$ acts on the $n$ embeddings of $L_f$ into $K_f$, its Galois closure. Let $\wp\subseteq\mathcal{O}_K$,$$\wp\mathcal{O}_{L_f}=\prod_i \mathfrak{P}_i^{e_i},$$where for each $i$, $\mathfrak{P}_i$ has inertia degree $f_i$ over $K$. Now, the primes dividing the discriminant of $L_f/K$ are either tamely ramified or wildly ramified.\begin{enumerate}
		\item[$\bullet$] If $\wp$ is tamely ramified,
		the inertia group $I_\wp$ is cyclic, and any generator $g\in G_f$ is the product of disjoint cycles consisting of $f_1$ clycles of length $e_1$, $f_2$ cycles of length $e_2$ and so on. Hence the exponent of $\wp$ dividing $\mathfrak{D}_{L_f/K}$ is$$v_\wp(\mathfrak{D}_{L_f/K})=\sum_i(e_i-1)f_i=\ind(g)\ge\ind(G_f)=k.$$
		\item[$\bullet$] If $\wp$ is wildly ramified, we have the strict inequalities$$v_\wp(\mathfrak{D}_{L_f/K})>\sum_i(e_i-1)f_i>k.$$
	\end{enumerate}In both cases we see that $\wp^k|\mathfrak{D}_{L_f/K}$.
\end{proof}
For a primitive group $G$, the followings are standard facts.\begin{enumerate}
	\item[a.] If $G$ contains a transposition, then $G=S_n$. In particular $\ind(G)\ge2$.
	\item[b.] If $G$ contains a 3-cycle or a double transposition and $n\ge9$, then $G=A_n$ or $S_n$. In particular $\ind(G)\ge3$.
\end{enumerate}	
It follows from Proposition 5, a and b that:
\begin{coroll}
	Let $f\in\mathcal{O}_K[X]$ be a monic, irreducible polynomial of degree $n$ with $G_f\subset S_n$ primitive. Then $\mathfrak{D}_{L_f/K}$ has the property $\Pm_2$.\\
	If moreover $G_f\neq A_n$ and $n\ge 9$, then $\mathfrak{D}_{L_f/K}$ has the property $\Pm_3$.
\end{coroll}	

We now follow and generalize the argument of Bhargava \cite{Bh1} by dividing the set $\mathscr{N}_n(N,G)$ into three sets.\\
For an irreducible $f\in\mathscr{N}_n(N,G)$ with $G$ primitive, let$$\mathfrak{C}_f:=\prod_{\wp|\mathfrak{D}_{L_f/K}}\wp$$ and denote by $\mathfrak{D}_f$ the discriminant $\mathfrak{D}_{L_f/K}$. \\
Let$$\mathscr{N}_n(N):=\underset{\tiny\mbox{primitive}}{\bigcup_{G\subset S_n}}\mathscr{N}_n(N,G).$$As observed before, we can assume that all polynomials are irreducible.\\
For $\delta>0$, the sets $\mathscr{N}_1(N,\delta),\mathscr{N}_2(N,\delta)$ and $\mathscr{N}_3(N,\delta)$ are defined as\begin{align*}
	\mathscr{N}_1(N,\delta)&:=\{f\in\mathscr{N}_n(N):|N_{K/\Q}\mathfrak{C}_f|\le N^{d(1+\delta)},\ |N_{K/\Q}\mathfrak{D}_f|> N^{d(2+2\delta)}\},\\
	\mathscr{N}_2(N,\delta)&:=\{f\in\mathscr{N}_n(N):|N_{K/\Q}\mathfrak{D}_f|< N^{d(2+2\delta)}\},\\
	\mathscr{N}_3(N,\delta)&:=\{f\in\mathscr{N}_n(N):|N_{K/\Q}\mathfrak{C}_f|> N^{d(1+\delta)}\}.	
\end{align*}
We use the following result, in which we identify the space of binary $n$-ic forms over $\mathcal{O}_K$ having leading coefficient 1 with the space of monic polynomials of degree $n$ over $\mathcal{O}_K$. The proof uses Fourier analysis over finite fields.\\
The index of a binary $n$-ic forms $f$ over $\mathcal{O}_K$ modulo $\wp|p$ is defined to be$$\sum_{i=1}^{r}(e_i-1)f_i,$$where $f\mod\wp=\prod_{i=1}^{r}P_i^{e_i}$, $P_i$ irreducible of degree $f_i$ over $\in\mathbb{F}_{p^{[\mathcal{O}_K/\wp:\mathbb{F}_p]}}$ for all $i$.

The proofs of the next results which are not included here, can be found in \cite{Bh1}.
\begin{prop}
	Let $0<\delta\ll_{n,d}1$ be small enough and let $\mathfrak{C}=\wp_1\dots\wp_m$, $\wp_i\neq\wp_j$ ($i\neq j$) be a product of primes in $\mathcal{O}_K$ of norm $|N_{K/\Q}\mathfrak{C}|<N^{d(1+\delta)}$.\\
	For each $i=1,\dots,m$ pick an integer $k_i$. Then the number of $K$-integral binary $n$-ic forms in a box $[-N,N]^{d(n+1)}$ with coefficients of height $\le N$, such that, modulo $\wp$, have index at least $k_i$, is at most$$\ll_{K,\varepsilon}\frac{N^{nd+\varepsilon}}{\prod_{i=1}^{m}|N_{K/\Q}\wp_i|^{k_i}}$$for every $\varepsilon>0$.
\end{prop}

Theorem 1, (3) follows by the three lemmas below together with Section 1.1.
\begin{lemma}For $\delta>0$ sufficiently small,
	$$|\mathscr{N}_1(N,\delta)|\ll_{n,K}N^{d(n-1)}$$as $N\rightarrow+\infty$.
\end{lemma} 
\begin{proof}
	Given a number field $L/K$, let $\mathfrak{C}$ be the product of the ramified primes and let $\mathfrak{D}$ be its discriminant. The polynomials $f$ so that $L_f\cong L$ (so $\mathfrak{C}_f=\mathfrak{C}$ and $\mathfrak{D}_f=\mathfrak{D}$) must have at least a triple root or at least two double roots modulo $\wp$ for every $\wp|\mathfrak{C}_f$. This follows easily by Proposition 6. Now, the density of the degree $n$ polynomials over a finite field $\mathbb{F}_q$ having a triple root is $1/q^2$, whereas the density of the ones having two double roots is $2/q^3$. Therefore the density of the above polynomials is$$\ll\prod_{\wp|\mathfrak{C}_f}\frac{2}{|N_{K/\Q}\wp|^2}\ll\frac{2^{\omega(\mathfrak{D})}}{|N_{K/\Q}\mathfrak{D}|},$$where $\omega(\mathfrak{D})$ is the number of prime divisors of $\mathfrak{D}$.\\
	By Proposition 6 the number of $f\in\mathscr{P}_{n,N}$ with $|N_{K/Q}\mathfrak{C}|\le N^{d(1+\delta)}$ and $\mathfrak{D}_f=\mathfrak{D}$ is$$\ll_{K,\varepsilon}\frac{N^{nd+\varepsilon}}{|N_{K/\Q}\mathfrak{D}|}.$$Summing over all $\mathfrak{D}$ of norm $|N_{K/\Q}\mathfrak{D}|>N^{d(2+2\delta)}$ gives\begin{multline*}
		\sum_{\mathfrak{D}}O_{K,\varepsilon}(N^{nd+\varepsilon}2^{\omega(\mathfrak{C})}/|N_{K/\Q}\mathfrak{D}|)\\
		=O_{K,\varepsilon}(N^{nd+\varepsilon}\cdot 2^{d(1+\delta)}\cdot N^{-2d-2d\delta})\\
		\ll_{n,K} N^{d(n-1)}.
	\end{multline*}
\end{proof}
\begin{lemma}
	If $n\le[3d+\sqrt{9d^2-4d}]+1$ then for $\delta>0$ sufficiently small,
	$$|\mathscr{N}_2(N,\delta)|\ll_{n,K}N^{d(n-1)},$$as $N\rightarrow+\infty$.
\end{lemma}
\begin{proof}
	Note that one can prove Theorem 4 by using a different bound, if holds, for the discriminat insted of $\ll N^{d(2n-2)}$ and improve the result itself. For the polynomials in our set we thus have$$|\mathscr{N}_2(N,\delta)|\ll_{n,K}N^{d\big(1+\frac{(2+2\delta)(n+2)}{4}\big)-\frac{2}{n}}\cdot(\log N)^{nd-1},$$which is at most$$\ll_{n,K}N^{d\big(1+\frac{(2+2\delta)(n+2)}{4}\big)-\frac{2}{n}+\varepsilon}$$for any $\varepsilon>0$. If $n\le[3d+\sqrt{9d^2-4d}]+1$, one has the desired upper bound $O_{n,K}(N^{d(n-1)})$.
\end{proof}
\begin{prop}
	Let $\wp\in\mathcal{O}_K$ be a prime ideal over $p$ and let $q=p^{[\mathcal{O}_k/\wp:\mathbb{F}_p]}$. If $h(X_1,\dots,X_n)\in\mathcal{O}_K[X_1,\dots,X_n]$ is such that\begin{align*}
		&h(c_1,\dots,c_n)\equiv 0\mod q^2,\\
		&h(c_1+qd_1,\dots,c_n+qd_n)\equiv 0\mod q^2
	\end{align*}for all $(d_1,\dots,d_n)\in\mathcal{O}_K^n$, then$$\frac{\partial}{\partial x_n}h(c_1,\dots,c_n)\equiv0\mod q.$$
\end{prop}
\begin{proof}
	Write\begin{multline*}
		h(c_1,\dots,c_{n-1},X_n)=h(c_1,\dots,c_n)+\frac{\partial}{\partial x_n}h(c_1,\dots,c_n)(X_n-c_n)\\+(X_n-c_n)^2r(X)
	\end{multline*}where $r(X)\in\mathcal{O}_K[X]$.  If we set $X_n$ to be in $\mathcal{O}_K$, $d_n\equiv c_n$ mod $\wp$, then the first and last terms are multiples of $\wp^2$, hence the middle term must be as well. Therefore $\frac{\partial}{\partial x_n}h(c_1,\dots,c_n)$ must be zero modulo $\wp$.
\end{proof}
\begin{lemma}
	For $\delta>0$ sufficiently small,
	$$|\mathscr{N}_3(N,\delta)|\ll_{n,K}N^{d(n-1)}$$as $N\rightarrow+\infty$.
\end{lemma}
\begin{proof}
	As in Lemma 8, for every $\wp|\mathfrak{C}_f=\mathfrak{C}$, $f$ has either at least a triple root or at least a pair of double roots modulo $\wp$. Let $q$ so that $f\mod \wp\in\mathbb{F}_q[X]$. Apply Proposition 7 to $d_f\mod\wp$ for every $\wp|\mathfrak{C}$ as a polynomial in the coefficients $\alpha_{n-1},\dots,\alpha_0$ of $f$. It follows that$$\frac{\partial}{\partial \alpha_0}d_f\equiv0\mod\mathfrak{C};$$hence so is the Sylvester resultant$$\mbox{Res}_{\alpha_0}(d_f,\frac{\partial}{\partial \alpha_0}d_f)=\pm d_{d_f(\alpha_0)}.$$Let $D(\alpha_{n-1},\dots,\alpha_1):=d_{d_f(\alpha_0)}$. Note that $D$ is not identically zero, thanks to the formulae for iterated discriminants of \cite{LMc}. Moreover, by Lemma 3.1 of \cite{Bh2}, the number of $\alpha_{n-1},\dots,\alpha_1$ in $\mathcal{O}_K$ of height $\le N$ so that $D(\alpha_{n-1},\dots,\alpha_1)=0$ is $O(N^{d(n-2)})$; the number of $f$ with such $\alpha_{n-1},\dots,\alpha_1$ is thus $O(N^{d(n-1)})$.

	Fix now $\alpha_{n-1},\dots,\alpha_1$ so that $D(\alpha_{n-1},\dots,\alpha_1)\neq0$. Then $D(\alpha_{n-1},\dots,\alpha_1)\equiv0\mod\mathfrak{C}$ for at most $O_{K,\varepsilon}(N^{\varepsilon})$ ideal factors $\mathfrak{C}$ of norm $N_{K/\Q}\mathfrak{C}>N^{d}$. Once $\mathfrak{C}$ is determined by $\alpha_{n-1},\dots,\alpha_1$ up to $O(N^{\varepsilon})$ possibilities, the number of solutions for $\alpha_0\mod\mathfrak{C}$ to $d_f\equiv0\mod\mathfrak{C}$ is $(\deg_{\alpha_0}(d_f))^{\omega(\mathfrak{C})}\ll_{K,\varepsilon}N^{\varepsilon}$. This is due to the fact that the number of solutions of $\alpha_0\mod\wp$ so that $d_f\equiv0\mod\wp$ for all $\wp|\mathfrak{C}$ is $\deg_{\alpha_0}(d_f)$.

	Since $N_{K/\Q}\mathfrak{C}>N^d$ the possibilities for $\alpha_0$ of height $\le N$ are also $O_{K,\varepsilon}(N^{\varepsilon})$. So the total number of $f$ is $O_{K,\varepsilon}(N^{d(n-1)+\varepsilon})$.
	
	We are going to remove the factor $N^{\varepsilon}$. To do this, consider$$\mathfrak{A}:=\underset{N_{K/\Q}\wp>N^{d\delta/2}}{\prod_{\wp|\mathfrak{C}}}\wp.$$
	\begin{enumerate}
		\item[$\bullet$] If $N_{K/\Q}\mathfrak{A}\le N^d$, then $\mathfrak{C}$ has a factor $\mathfrak{B}$ of norm$$N^{d\left( 1+\frac{\delta}{2}\right) }\le N_{K/\Q}\mathfrak{B}\le N^{d(1+\delta)},$$with $\mathfrak{A}|\mathfrak{B}|\mathfrak{C}$. Let $\mathfrak{B}$ be such a factor of largest norm. Define$$\mathfrak{D}':=\prod_{\wp|\mathfrak{B}}\wp^{v_{\wp}(\mathfrak{D})}.$$Then $N_{K/\Q}\mathfrak{D}'>N^{d(2+\delta)}$. The same argument of Lemma 8 with $\mathfrak{B}$ in place of $\mathfrak{C}$ and $\mathfrak{D}'$ in place of $\mathfrak{D}$ gives the estimate$$\sum_{N_{K/\Q}\mathfrak{D}'>N^{d(2+\delta)}}O_{K,\varepsilon}(N^{nd+\varepsilon}\cdot N^{-2d-d\delta})\ll_{n,K}N^{d(n-1)}.$$
		\item[$\bullet$] If $N_{K/\Q}\mathfrak{A}> N^d$, we use the original argument at the beginning of the proof with $\mathfrak{A}$ in place of $\mathfrak{C}$. We have that $\mathfrak{A}$ is a divisor of $D(\alpha_{n-1},\dots,\alpha_1)$. Let $\alpha_{n-1},\dots,\alpha_1$ so that $D(\alpha_{n-1},\dots,\alpha_1)\neq0$.
		
		Now, $d_f(\alpha_0)$ is a polynomial in $\alpha_0$ of degree $\le 2n-2$; its coefficients are monomials in $\alpha_{n-1},\dots,\alpha_1$ of degree $\le 2n-2$. Therefore $D$, whose degree is $\le 4n-6$, has bounded norm$$N_{K/\Q} D\ll N^{d(2n-2)(4n-6)}.$$. The number of primes $\wp$ with $N_{K/\Q}\wp> N^{d\delta/2}$ dividing $D$ is then at most$$\ll\frac{\log(N^{d(2n-2)(4n-6)})}{N^{d\delta/2}}\ll_{n,d}1.$$Once $\mathfrak{A}$ is determined by $\alpha_{n-1},\dots,\alpha_1$, the number of solutions for $\alpha_0\mod\mathfrak{A}$ to $d_f\equiv0\mod\mathfrak{A}$ is $O_{n,K}(1)$. Since $N_{K/\Q}\mathfrak{A}>N^d$, the total number of $f$ is then $O_{n,K}(N^{d(n-1)})$.
	\end{enumerate}
\end{proof}

\section{Proof of Theorem 2}
\subsection{Counting $G$-polynomials over $K$}
\begin{lemma}
	Let $n>r$, $(n,r)=1$, $\alpha_1,\dots,\alpha_{r-1},\alpha_{r+1},\dots,\alpha_{n-1}\in\mathcal{O}_K$ be fixed. Then$$X^n+\alpha_{n-1}X^{n-1}+\dots+\alpha_1X+t\in(\mathcal{O}_K[t])[X]$$has for all but at most $O_{n,d}(1)$ $\alpha_{n-r}$ in $\mathcal{O}_K$ the full $S_n$ has Galois group of $K(t)$.
\end{lemma}
\begin{proof}
	This follows from Satz 1 of \cite{He}.
\end{proof}
\begin{lemma}
	Let $f(X)=X^n+\alpha_{n-1}X^{n-1}+\dots+\alpha_0\in\mathcal{O}_K[X]$ with roots $\beta_1,\dots,\beta_n\in K_f$ and $G_{K_f/K}=G\subseteq S_n$. Let$$\Phi(z;\alpha_0,\dots,\alpha_{n-1})=\prod_{\sigma\in S_n/G}\Big(z-\sum_{\tau\in G}\beta_{\sigma\tau(1)}\beta^2_{\sigma\tau(2)}\dots \beta^n_{\sigma\tau(n)}\Big)$$be the Galois resolvent with respect to $\sum_{\tau\in G}X_{\tau(1)}X^2_{\tau(2)}\dots X^n_{\tau(n)}$. Then $\Phi$ has integral coefficients and the roots are integral over $K$.
\end{lemma}
\begin{proof}
	The polynomial $\Phi$ is fixed by any permutation of the roots. Then the coefficients are symmetric polynomials in the roots of $f$, hence they can be written as integral polynomials in the elementary symmetric polynomials of the roots of $f$, that is in the coefficients of $f$.\\
	The root $\sum_{\tau\in G}\beta_{\sigma\tau(1)}\beta^2_{\sigma\tau(2)}\dots \beta^n_{\sigma\tau(n)}$ of $\Phi$ is fixed by any element of $G$, so it is in $K$. It also satisfy a monic polynomial with coefficients in $\mathcal{O}_K$. Then it is integral over $K$.
\end{proof}
\begin{lemma}
	Let $F\in\mathcal{O}_K[X_1,X_2]$ of degree $n$ be irreducible over $K$. For $P_i\in\R_{\ge1}$, $i=1,2$, let$$N(F;P_1,P_2)=|\{(x_1,x_2)\in\mathcal{O}_K^2:F(x_1,x_2)=0,\ \h(x_i)\le N\ i=1,2\}|.$$Denote by$$T=\max_{(e_1,e_2)}\{P_1^{de_1},P_2^{de_2}\},$$where the maximum takes over all integer 2-uples $(e_1,e_2)$ for which the corresponding monomial $X_1^{e_1}X_2^{e_2}$ occurs in $F(X_1,X_2)$ with nonzero coefficient. Then for every $\varepsilon>0$ $$N(F;P_1,P_2)\ll_{n,d,\varepsilon}\max\{P_1,P_2\}^{\varepsilon}\cdot\exp\left(\frac{d^2\log P_1\log P_2}{\log T} \right).$$
\end{lemma}
\begin{proof}
	It is a straightforward generalization of the special case $P_1=1$ of Theorem 1 in \cite{BHB}. See also \cite{HB}, Theorem 15. As noticed in \cite{Di2}, if $F$ is irreducible over $K$, by B\'{e}zout's Theorem $N(F;P_1,P_2)\ll_{n,d}1$, so we may assume that $F$ is absolutely irreducible, as in \cite{BHB}.
\end{proof}
We can now prove Theorem 2. Let $G$ be a subgroup of $S_n$ of index $[S_n:G]=m$. By Lemma 12, there exist $b_1,\dots,b_m\in\Z[\alpha_0,\dots,\alpha_{n-1}]$ so that$$\Phi(z;\alpha_0,\dots,\alpha_{n-1})=z^m+b_1(\alpha_0,\dots,\alpha_{n-1})z^{m-1}+\dots+b_m(\alpha_0,\dots,\alpha_{n-1}).$$ By Lemma 1, a root $z\in\mathcal{O}_K$ of $\Phi$ has norm bounded by$$|N_{K/\Q}|\ll_{n,d}N^{d\alpha}$$for some $\alpha\ge1$.

Now fix $\alpha_{n-1},\dots,\alpha_2$ of height $N$. Our goal is to bound the number of $\alpha_1,\alpha_0\in\mathcal{O}_K$ of height $N$ so that $G_{K_f/K}=G$. It suffices to show that there are at most $O(N^{d(1+1/m)+\varepsilon})$ such $\alpha_1,\alpha_0$.

By Lemma 11, $X^n+\alpha_{n-1}X^{n-1}+\dots+\alpha_1X+t$ has for all but at most $O_{n,d}(1)$ values of $\alpha_1$ the full symmetric group as Galois group over $K(t)$. Hence it's enough to fix any such $\alpha_1$ of height $N$ for which $X^n+\alpha_{n-1}X^{n-1}+\dots+\alpha_1X+t$ has Galois group $S_n$ over $K(t)$ and then show that for those fixed $\alpha_{n-1},\dots,\alpha_1$ there are at most $O(N^{d/m+\varepsilon})$ possibilities for $\alpha_0$, $\h(\alpha_0)\le N$, for which $f$ has Galois group $G$.

Consider $\Phi(z;\alpha_0,\dots,\alpha_{n-1})=\Phi(z,\alpha_0)$ as a polynomial in $z,\alpha_0$. Since $X^n+\alpha_{n-1}X^{n-1}+\dots+\alpha_1X+t$ has Galois group $S_n$, the resolvent $\Phi(z,\alpha_0)$ must be irreducible over $K[z]$. We can now bound above the number of zeros of $\Phi(z,\alpha_0)$ with $|N_{K/\Q}(z)|\ll N^{d\alpha}$ and $\h(\alpha_0)\le N$ by applying Lemma 13 with $P_1\asymp N^{\alpha}$ and $P_2=N$. In this case $T\gg N^{dm\alpha}$, so\begin{multline*}
	|\{(z,\alpha_0)\in\mathcal{O}_K^2: |N_{K/\Q}(z)|\ll N^{d\alpha},\ \h(\alpha_0)\le N,\ \Phi(z,\alpha_0)=0\}|\\
	\ll_{n,d,\varepsilon}N^{\varepsilon}\cdot\exp\left( \frac{d^2\log N^{\alpha}\log N}{dm\alpha\log N}\right)\\
	\ll_{n,d,\varepsilon}N^{\frac{d}{m}+\varepsilon}. 
\end{multline*}This completes the proof of Theorem 2.

\pagebreak

\noindent
\footnotesize
DEPARTEMENT MATHEMATIK, ETH EIDGEN\"{O}SSICHE TECHNISCHE HOCHSCHULE, 8092 Z\"{U}RICH, SWITZERLAND\\
\textit{Email address}: \texttt{ilaria.viglino@yahoo.it}

\end{document}